\documentclass[a4paper,leqno]{article}
\usepackage[english]{babel}
\usepackage[utf8]{inputenc}
\usepackage{amsmath}
\usepackage{amssymb}
\usepackage{graphicx}
\usepackage[all,dvips]{xy}
\usepackage{amsthm}
\usepackage{amsfonts}
\usepackage{color}
\usepackage[active]{srcltx}
\usepackage{lmodern}
\usepackage{enumerate}
\usepackage{fancybox}
\usepackage{amsbsy}
\usepackage{pstricks}
\usepackage{xspace}

\def\lam{\lambda}

\def\eps{\epsilon}
\def\be{\beta}

\def\sig{\sigma}
\def\om{\omega}
\def\Om{\Omega}
\def\Ga{\Gamma}
\def\tt{\theta}
\def\R{{\mathbb R}}

\def\N{{\mathbb N}}
\def\C{{\mathcal C}}
\def\D{{\mathcal D}}
\def\T{{\mathcal T}}
\def\d{{\partial}}
\def\ds{\displaystyle}
\def\div{{\rm div}}
\def\supp{{\rm supp}}
\def\sgn{{\rm sgn}}
\def\inf{{\rm inf}}
\def\sup{{\rm sup}}

\def\nuu{\overrightarrow{\nu}}

\def\Wd{W_\div}
\def\W{W^{1,1}((0,+\infty)\, ;\, L^1(\Ga))}
\def\Winf{W^{1,\infty}((0,\infty)\,;\,L^\infty(\Ga))}

\def\V{\tilda{V}}
\def\tbeta{\tilda{\beta}}
\def\trho{\tilda{\rho}}
\def\Gn{G\cdot \nuu}

\newtheorem{thm}{Theorem}[section]
\newtheorem{cor}{Corollary}[section]
\newtheorem{defn}{Definition}[section]
\newtheorem{prop}{Proposition}[section]
\newtheorem{lem}{Lemma}[section]
\newtheorem{rem}{Remark}

\newcommand{\ncmd}{\newcommand}
\ncmd{\tilda}{\widetilde}
\ncmd{\beq}{\begin{equation}} \ncmd{\eeq}{\end{equation}}
\ncmd{\beqn}{\begin{equation*}} \ncmd{\eeqn}{\end{equation*}}
\ncmd{\beqa}{\begin{eqnarray}} \ncmd{\eeqa}{\end{eqnarray}}
\ncmd{\beqan}{\begin{eqnarray*}} \ncmd{\eeqan}{\end{eqnarray*}}
\ncmd{\barre}[1]{\overline{ #1}}

\ncmd{\espace}{\hspace*{0.5cm}}

\begin{document}

\title{Mathematical analysis of a two-dimensional population model
of metastatic growth including angiogenesis.}

\author{Benzekry S\'ebastien\thanks{CMI-LATP, UMR 6632, Universit\'e
de Provence, Technop\^ole Ch\^ateau-Gombert, 39, rue F. Joliot-Curie,
13453 Marseille cedex 13, France. E-mail:
\texttt{benzekry@phare.normalesup.org}}$\;^{,}$\thanks{Laboratoire de Toxicocinétique et Pharmacocinétique UMR-MD3. 27, boulevard Jean Moulin 13005 Marseille. France.}}
\date{\today}
\maketitle

\begin{abstract}
\noindent
Angiogenesis is a key process in the tumoral growth which
allows the cancerous tissue to impact on its vasculature in order to
improve the
nutrient's supply and the metastatic process. In this paper, we
introduce a model for the density of metastasis which takes into
account for this feature. It is a two dimensional structured equation
with a vanishing velocity field and a source term on the boundary.
We present here the mathematical analysis of the model, namely the
well-posedness of the equation and the asymptotic behavior of the
solutions, whose natural regularity led us to investigate some basic
properties of the space $\Wd(\Om)=\{V\in L^1;\;\div(GV)\in L^1\}$,
where $G$ is the velocity field of the equation.


\end{abstract}

\tableofcontents
\newpage
\section{Introduction}

In the seventies, Judah Folkman puts forward the
assumption that a cancer tissue, like other tissues, needs nutrients
and oxygen conveyed by the blood vessels. Consequently, tumoral
growth and development of metastasis are dependent on angiogenesis, a
process consisting in building and developing the vascularization.
From this discovery, a new anti-cancer therapeutic way is open : to
starve cancer by depriving it of its vascularization. If for the
last two decades, more than ten antiangiogenics drugs have been
developed, mainly monoclonal antibodies and tyrosin kinase
inhibitors, the administration protocols are far from being
optimal.
It is enough for example to consult the publication \cite{ebos} to
realize the paroxystic effects they can induce.
\paragraph*{} Thus, a tool for in silico studying the
administration
protocols for antiangiogenic drugs could largely contribute to
optimize the effectiveness of the treatments, in particular to avoid
some therapeutic failures. In this direction, the construction of a
mathematical model taking into account the mechanisms of tumoral
angiogenesis and the effect of the antiangiogenics agents proves to
be an essential stage in order to improve the use of antiangiogenic
therapies. Some work (for example in \cite{ribba} and \cite{billy})
was made with the aim of qualitatively studying the effects of
antiangiogenic therapies on the control of the primitive tumor
growth. In this work we propose a modeling which purpose is to
describe the action of the currently used clinical protocols, not
only on the tumoral growth, but also on the production of metastases.
\paragraph{} The model is a combination of the PDE model for the
metastasis density proposed by \cite{iwata} and studied in
\cite{BBHV,devys}, with the ODE model for each metastasis' growth of
Folkman et al. \cite{folkman}. This transport equation endowed with a
non-local boundary condition expressing creation of metastasis
can be classified as part of the so-called structured population
equations arising in mathematical biology which have the following
general expression
\begin{eqnarray}\label{forme_generale}\left\lbrace\begin{array}{ll}
\d_t \rho + \div(F(t,X,\rho)) = - \mu(t,X,\rho) & \Om\\
-\Gn \rho(t,\sig) = B(t,\sig,\rho) & \sig \in \d
\Om\;s.t.\;\Gn(\sig) <0 \\
\rho(0,X)=\rho^0(X)             & \Om
\end{array}\right. .\end{eqnarray}
The introduction of such
equations in the linear case is due to Sharpe and Lotka in 1911
\cite{lotka} and
McKendrick in 1926 \cite{mckendrick}. Although these equations have
been widely studied both in the linear and nonlinear cases
(for an introduction to the linear theory see the book of Perthame
\cite{perthame} and to the nonlinear one see the book of Webb
\cite{webb}, as well as \cite{perthame_nonlinear} for a survey), a
complete
general theory has not been achieved yet, even in the linear case.
Indeed, most of the models have the so called structuring variable
$X$ being one-dimensional and often representing the age, thus
evolving with $F(t,a,\rho)=\rho$. A difficulty on the
regularity of solutions is introduced when the velocity is
non-constant and vanishes (see \cite{BBHV,devys}). Dealing with
situations in dimensions higher than one is not a common thing.
\paragraph{} In our case, the model is a linear equation, with
$$F(t,X,\rho)=G(X)\rho$$
structured in two variables : $X=(x,\tt)$ with $x$ the size of
metastasis and $\tt$ the so-called ``angiogenic capacity''.
The velocity field $G$ vanishes on the boundary of the domain, which
is a square. Moreover, we have an additional source term in the
boundary condition of the equation :
$$-\Gn \rho(t,\sig)=N(\sig)\int \beta(X) \rho(t,X)
+f(t,\sig).$$
As far as we now, the mathematical analysis for multi-dimensional
models is done only in situations where one of the structured
variables is
the age and thus with the first component of $G$ being constant (see
for instance \cite{tucker, crauste, doumic}). In the context of the
follicular control during the ovarian process, a nonlinear model
structured in dimension two with both components of the velocity field
$G$ being non-constant is introduced in \cite{echenim} but no
mathematical analysis is performed due to the complexity of the model.
\paragraph{}In the present paper, we address the problem of
the mathematical analysis of our model, namely : existence,
uniqueness, regularity and asymptotic behavior of the solutions.
Following the method used in \cite{BanksKappel} and \cite{BBHV}, we
use a semigroup approach to deal with the existence and regularity of
the solutions. The main difficulties we have to deal with in this two
dimensional problem come from the singularity of the velocity field,
as well as the presence of a time-dependent source term in the
boundary condition. During the study, we take a
particular attention on the problems of regularity of the
solutions and approximation of weak solutions by regular ones, which
led us to study the space $\Wd(\Om)$ (see the appendix) . The
paper is organized as follows : in the section \ref{model} we present
the model, in the section \ref{properties_operator} we study the
properties of the underlying operator and in the section
\ref{existence_asymp} we apply our study to the evolution equation
from our model.

\section{Model}\label{model}

The model we developed is an improvement of the model
proposed by \cite{iwata} and studied in \cite{BBHV}. We want now to
take into account the key process of angiogenesis in the tumoral
growth and integrate it in the metastatic evolution. To do this, we
combine a renewal equation describing the evolution of the density of
metastasis with an ODE model of tumoral growth including angiogenesis
developed by Hahnfeldt et al. in \cite{folkman}.


 \subsection{The ODE model of tumoral growth under angiogenic
control}\label{model_ODE}
We present now the model of Hahnfeldt et al.
from \cite{folkman}. Let $x(t)$ denote the size of a given tumor at
time $t$. The growth of the tumor is modeled by a gompertzian growth
rate, which expression is :
\begin{equation}\label{g1}
g_1(x)=ax\ln\left(\frac{\theta}{x}\right),
\end{equation}
where $a$ is a parameter representing the velocity of the growth and
$\tt$ the carrying capacity of the environment. The idea is now to
take $\tt$ as a variable of the time, representing the degree of
vascularization of the tumor and called "angiogenic
capacity". The variation rate for $\tt$ derived in \cite{folkman} is
:
\begin{equation}\label{g2}
g_2(x,\tt)=cx-d\tt {x}^{\frac{2}{3}},
\end{equation}
If we denote
$X(t)=\left(x(t),\tt(t)\right)$ and
define $G(X)=\left(g_1(x,\tt), g_2(x,\tt)\right)$
we have the following system of ODE modeling the tumoral growth :
\begin{eqnarray}\left\lbrace \begin{array}{l}\label{systeme2}
\frac{dX}{dt}=G(X) \\
X(t_0)=\left(
\begin{array}{c}
x^0 \\
\tt^0
\end{array}
\right)
\end{array}
\right.  \end{eqnarray}
In the figure \ref{caracteristiques}, we present some numerical
simulations of the phase
plan of the system.

\begin{figure}[!h]
\begin{center}
\includegraphics[width=8cm]{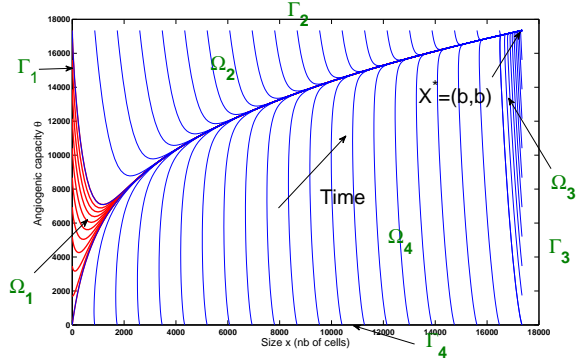}
\caption{Phase plan of the ODE system
\eqref{systeme2}.}
\label{caracteristiques}
\end{center}
\end{figure}

This system has been studied by A. d'Onofrio and A.
Gandolfi in \cite{dOnofrio_Gandolfi}. We define
$$b=\left(\frac{c}{d}\right)^{\frac{3}{2}},\; \Om=(1,b)\times (1,b),\;
\Ga=\d \Om$$
\paragraph{}We will now turn our interest to the flow defined by the
solutions of the system of ODE, as it will play a fundamental role in
the sequel. We define the application
$$\Phi:\begin{array}{ccc} [0, \infty[\times \Gamma & \rightarrow
& \barre{\Om} \\
	                                 (\tau, \sig  )&
\mapsto & \Phi_\tau(\sig) \end{array}$$
as being the solution of the system \eqref{systeme2} at time $\tau$
with the initial condition $\sig$. We use of this application in order
to
see $\Om$ as $\Om \simeq \;]0,\infty[\times \Ga.$
More precisely, we will show that $\Phi$ is an homeomorphism locally
bilipschitz $]0,\infty[ \times\Ga^*\rightarrow \Om$
where $\Ga^*:=\Ga\backslash\{(b,b)\}.$
In order to have a candidate for the inverse of $\Phi$,
we define for $(x,\tt)\in \Om$
$$\begin{array}{c}
\tau(x,\tt)=\rm{inf}\{\tau \geq 0 |\; \Phi_{-\tau}(x,\tt) \in \Ga^*\}
\\
\sig(x,\tt)=\Phi_{-\tau(x,\tt)}(x,\tt)
\end{array}$$
The qualitative properties of the ODE imply the existence of such a
couple $(\tau,\sig)$ (the field points inward along
$\Ga^*$ and the solutions all converge to $X^*$ (see
\cite{dOnofrio_Gandolfi}) so going back in time
they meet the boundary), and the Cauchy-Lipschitz theorem implies
uniqueness because the system is autonomous and thus the
characteristics don't cross each other in the phase plane. The time
$\tau(x,\tt)$ is the time spent in $\Om$ and $\sig(x,\tt)$ is the
entrance point of the characteristic passing through the point
$(x,\tt)$. From the Lipschitz regularity of $\Om$ we can't expect
$\Phi$ to be globally $\C^1$, this is why we introduce the following
open sets :
$$\Om_i=\{\Phi_\tau(\sig); \; \sig \in \Ga_i, \, \tau\in ]0,\infty[\},
\quad i=1,2,3,4$$
where
$$\Ga_1=](1,1) ,(1,b)[, \; \Ga_2=](1,b) ,(b,b)[, \; \Ga_3=](b,b)
,(b,1)[, \; \Ga_4=](b,1) ,(1,1)[$$
The restriction of $\Phi$
to $]0,\infty[\times \Ga_i$ is a
diffeomorphism, as established in the following proposition.

 \begin{prop}[Properties of the flow]\label{proprietes_flot}~
\newline (i) The application $\Phi$ is a diffeomorphism
$]0,\infty[\times \Ga_i
\rightarrow \Om_i$ and for every $\tau\geq 0$ and almost every $\sig
\in \Ga$
 \begin{equation}\label{expression_J}
 J_\Phi(\tau,\sig)=G\cdot \nuu(\sig) e^{\int_0^\tau
\div(G(\Phi_s(\sig)))ds}
 \end{equation}
 where $J_\Phi(\tau,\sig)$ is the Jacobian of $\Phi$.
 \newline (ii) Globally, $\Phi$ is an homeomorphism $]0,\infty[\times
\Ga^* \rightarrow \Om$ locally bilipschitz.
 \end{prop}

 \begin{rem}
 The regularity proven here on $\Phi$ validates the use of $\Phi$ as
a
change of variables (see \cite{droniou1} for locally Lipschitz
changes of variables).
 \end{rem}

\begin{proof}~\\
\espace $\bullet$ \emph{$\Phi$ is one-to-one and onto}. Let
$X=(x,\tt) \in \Om$. We have
$\Phi(\tau(X),\sig(X))=X$ because
$\Phi_{-\tau(X)}(X)=\sig(X)$ implies $X=\Phi_{\tau(X)}(\sig(X))$
(indeed $\Phi_{-\tau}$ is the inverse of $\Phi_\tau$ when $\tau$ is
fixed). In the same way,
$(\tau(\Phi_\tau(\sig)),\sig(\Phi_\tau(\sig)))=(\tau,\sig)$. Thus
$\Phi$ is one-to-one and onto and
$\Phi^{-1}(x,\tt)=(\tau(x,\tt),\sig(x,\tt))$. \\
\espace $\bullet$ \emph{$\Phi$ is a diffeomorphism on
$]0,\infty[\times \Ga_i$}. Using the general theorem of dependency on
the initial conditions for ODEs, $\Phi$ is $\C^1(]0,\infty[\times
\Ga_i)$ and if we call $\sig(s)$ a parametrization of $\Ga_i$, we have
$\frac{\d\Phi}{\d s}(\tau,\sig(s))=D_y\Phi_\tau(\sigma(s)) \circ
\sigma'(s)$, and the following characterization of
$\frac{\d\Phi}{\d s}(\tau,\sig (s))$ stands : for each $s$, it is
the solution of the differential equation
$$\left\lbrace\begin{array}{l}
   \frac{dZ}{d\tau}=DG(\Phi) \circ Z \\
   Z(0)= \sigma'(s)
   \end{array}\right.$$
Using this characterization, we can derive the formula
\eqref{expression_J} for the Jacobian $J_\Phi(\tau,\sig)$. We have
$J_\Phi(\tau,\sig)=\frac{\d\Phi}{\d s} \wedge \frac{\d\Phi}{\d
\tau}=\frac{\d\Phi}{\d s} \wedge G(\Phi)$, and differentiating
in $\tau$, we get
\begin{align*}
       \frac{\d}{\d \tau}J_\Phi(t,\sig) & = DG \circ \frac{\d\Phi}{\d
s}
\wedge G(\Phi) + \frac{\d\Phi}{\d s} \wedge DG \circ G(\Phi)\\
					 & = {\rm trace}(DG)
J_\Phi(t,\sig) = \div(G) J_\Phi(t,\sig)
       \end{align*}
Hence, for all $\sig(s)$, using that $J_\Phi(0,\sig(s))=\sigma'(s)
\wedge G(\sig(s)) = |\sigma'(s)| G \cdot \overrightarrow{\nu}(\sig(s))
\neq 0,$ we obtain the formula
\begin{equation}\label{expression_J2}
J_\Phi(t,\sig(s))= |\sigma'(s)| G\cdot \overrightarrow{\nu}(\sig(s))
\exp(\ds\int_0^t\div(G(\Phi(\tau,\sig(s))))d\tau) \neq 0
\end{equation}
We get \eqref{expression_J} by choosing a parametrization
with  velocity equal to one.
\begin{rem}
In the sequel, we fix this parametrization\end{rem}
We can then apply the global inversion theorem to conclude that $\Phi$
is a $\C^1$-diffeomorphism
$]0,\infty[ \times \Ga_i \rightarrow \Om_i$. \\
\espace $\bullet$ \emph{Globally}. From the given properties of the
vector field $G$, we can extend the flow to a neighborhood $V$ of
$\barre{\Om}$, and we have that it is $\C^1(]0,\infty[\times V)$ (see
\cite{demailly}, XI p.305). Hence $\Phi$, which is the
restriction of this application to $]0,\infty[\times \Ga^*$ with
$\Ga^*$ being Lipschitz, is locally Lipschitz. Remark here that it is
not globally Lipschitz since $\frac{\d}{\d \sig}\Phi_\tau(\sig)$
can blow up when $\tau$ goes to infinity, due to the singularity at
$X^*$.\\
\espace To show that $\Phi^{-1}$ is also locally
Lipschitz on $\Om$ we consider some compact set $K \subset \Om$ and
show that $\Phi^{-1}$ is Lipschitz on $K$. We define $K_i =
\barre{\Om_i}\cap K$, and
$\tilda{K_i}:=\Phi^{-1}(K_i) \subset ]0,\infty[\times \barre{\Ga_i}$.
Now since $\Phi$ is the restriction of a globally $\C^1$ application,
we have $\Phi \in \C^1(\tilda{K_i})$, meaning that its differential
$D\Phi$ is continuous until the boundary of $\tilda{K_i}$.
Moreover using the formula \eqref{expression_J2}, we see that the
value of $D\Phi$ on $\d \tilda{K_i}$ is invertible since we avoid the
singularity $X^*$. Hence, using the continuity of the inverse
application we obtain that $D\Phi^{-1}=(D\Phi)^{-1}$ is continuous on
$K_i$. Thus $\Phi^{-1} \in \C^1(K_i)$ and so it is Lipschitz on each
$K_i$. As the global continuity of $\Phi^{-1}$ on $\Om$ is deduced
from the continuity on $\Om$ of $X  \mapsto  \tau(X),$ it is Lipschitz
on $K$.
\end{proof}


 \subsection{A renewal equation for the density of metastasis}
Starting from the velocity field $G$ of the previous
subsection for one given tumor, we now derive a renewal equation for
the density $\rho(t,x,\tt)$ of metastasis at time $t$, size (= number
of cells) $x$ and so called "angiogenic capacity" $\tt$. The term
density for $\rho$ means that the number of metastasis at time $t$ in
an infinitesimal volume centered in $(x,\tt)$ and of size $dxd\tt$ is
$\rho(t,x,\tt)dxd\tt$. We assume that each metastasis evolves in the
space $(x,\tt)$ with the velocity $G(x,\tt)$. Expressing the
conservation of the number of metastasis, we obtain
\begin{equation}\label{transport}
\d_t \rho + \div(\rho G) =0 .
\end{equation}
\paragraph{} The metastasis cannot have size nor angiogenic capacity
bigger than the parameter $b$, and we assume them
to have size and angiogenic capacity bigger than $1$. We are thus
driven to consider the transport equation \eqref{transport} in the
square $\Om=(1,b)\times(1,b)$. The field $G$ pointing inward
all along the boundary, we need now to precise the boundary condition
on $\Ga$.\\
\espace Metastasis do not only grow in size and angiogenic
capacity, they are also able to emit new metastases. We denote by
$\beta(x,\tt,\sig)$ the birth rate of new metastasis with
size and angiogenic capacity $\sig\in \Ga$ by metastasis of size $x$
and angiogenic capacity $\tt$, and by $f(t,\sig)$ the term
corresponding to metastasis produced by the primary tumor. Expressing
the equality between the entering flux of metastasis and the number
of new born, we derive the following boundary condition on $\Ga$ :
$$-\Gn (\sig) \rho(t,\sig)=\int_\Om \beta(x,\tt,\sig) \rho(t,x,\tt)
dxd\tt + f(t,\sig)$$
We then assume that there is no coupling between $(x,\tt)$ and $\sig$
in the expression of $\beta$, which is traduced by an expression of
$\beta$ as $\beta(x,\tt,\sig)=N(\sig) \beta(x,\tt)$.
Now let $\;Q:=]0,+\infty[ \times \Omega,\; \Sigma :=]0,+\infty[\times
\Ga$. The equation is
 \begin{eqnarray} \label{equation}\left\lbrace \begin{array}{ll}

	   \d_t \rho + \div(G\rho) =0     & \forall \,(t,x,\tt)\in Q
\\

	   -G\cdot \nuu \rho (t,\sig) = N(\sig) \int_\Om \beta(x,\tt)
\rho(t,x,\tt) dxd\tt + f(t,\sig) &
\forall \, (t,\sig) \in \Sigma \\

	   \rho(0,x,\tt)=\rho^0(x,\tt)  & \forall \, (x,\tt) \in \Om
	
\end{array}
\right.
\end{eqnarray}
We will do the following assumptions on the data :
\begin{align}\label{hypotheses_donnees}
\beta \in L^\infty,\;\beta \geq 0 \;a.e.,\; N \in
&Lip(\Ga) \text{ with compact support in $\Ga^*$}, \; N\geq 0,\;
\int_\Ga N =1 \nonumber\\
 &G \text{ given by \eqref{g1} and \eqref{g2}}
\end{align}
\begin{rem}
In practice, the new metastasis only appear with size $1$ and there
should not exist metastasis on $\Ga_{2,3,4}$, thus in the biological
model we have $\supp(N) \subset \Ga_1$. The expression of $\beta$ we
use in the biological applications is $\beta(x,\tt)=mx^\alpha$
with $m$ and $\alpha$ two positive parameters traducing respectively
the aggressiveness of the cancer and the spatial organization of the
vasculature. The source term $f$ has the following
expression in biological applications
: $f(t,\sig)=N(\sig)\beta(X_p(t))$
with $X_p(t)$ representing the primary tumor and
being solution of the system \eqref{systeme2}.
\end{rem}
\begin{defn}[Weak solution]\label{solutions_faibles}
Let $\rho^0 \in L^1(\Om)$ and $f\in L^1(]0,\infty[\times\Ga)$. We
call weak solution of the equation \eqref{equation} any function
$\rho \in \C([0,\infty[;L^1(\Om))$ which verifies: for every $T>0$
and every function
$\phi \in\C^1_c([0,+\infty[\times \barre{\Om}^*)$
\begin{align}\label{identite_distributions2}
&\int_0^T \int_\Om \rho[\d_t\phi + G\cdot \nabla \phi]dtdxd\tt  +
\int_\Om \rho^0(x,\tt)\phi(0,x,\tt)dxd\tt \\
	      &- \int_\Om \rho(T,x,\tt)\phi(T,x,\tt)dxd\tt -
\int_0^T\int_\Ga N(\sig)
\left(\int_\Om \beta(x,\tt)\rho(t,x,\tt) dxd\tt\right)
\phi(t,\sig)d\sig dt=0 \nonumber
\end{align}
\end{defn}
\paragraph{}Analyzing the equation \eqref{equation} indicates that the
solution
is the sum of two terms : an homogeneous one associated to
the initial condition, which solves the equation without the source
term $f$ (which we will refer to as the homogeneous equation)
\begin{eqnarray} \label{equation_homogene}\left\lbrace
\begin{array}{ll}

	   \d_t \rho + \div(G\rho) =0     & \forall \,(t,x,\tt)\in Q
\\

	   -G\cdot \nuu \rho (t,\sig) = N(\sig) \int_\Om \beta \rho(t)
dxd\tt &
\forall \, (t,\sig) \in \Sigma \\

	   \rho(0,x,\tt)=\rho^0(x,\tt)  & \forall \, (x,\tt) \in \Om
	
\end{array}
\right.
\end{eqnarray}
and a non-homogeneous term associated to the
contribution of the source term $f(t,\sig)$ and solution to the
equation (which will be refered as the non-homogeneous equation)
\begin{eqnarray} \label{equation_nonhomogene}\left\lbrace
\begin{array}{ll}

	   \d_t \rho + \div(G\rho) =0     & \forall \,(t,x,\tt)\in Q
\\

	   -G\cdot \nuu \rho (t,\sig) = N(\sig) \int_\Om \beta \rho(t)
dxd\tt + f(t,\sig) &
\forall \, (t,\sig) \in \Sigma \\

	   \rho(0,x,\tt)=0  & \forall \, (x,\tt) \in \Om  	
\end{array}
\right.
\end{eqnarray}
For existence and uniqueness of solutions, we will deal with the
homogeneous problem using the semigroup
theory and with the non-homogeneous one via a fixed point
argument.

 \subsection{Semigroup formulation for the homogeneous problem}
We reformulate \eqref{equation_homogene} as a Cauchy problem
\begin{eqnarray}\label{equation_sg}\left\lbrace \begin{array}{c} \d_t
\rho(t) = A\rho(t) \\
\rho(0)=\rho^0	\end{array}\right.
\end{eqnarray}
   We introduce the following space:
$$W_\div(\Om)=\{V\in L^1(\Om)|\;\div(GV)\in L^1(\Om)\},$$ and the
following operator $$\begin{array}{ccccc}
A &:&D(A)\subset L^1(\Om)&\rightarrow &L^1(\Om) \\
 & &V                   &\mapsto     &-\div(GV)
\end{array}, $$ where
\begin{equation} \label{domA} D(A)= \{V\in \Wd(\Om) ;\;
-G.\nuu\cdot\gamma(V)(\sig)=N(\sig)\int_\Om
\be(x,\tt)V(x,\tt)dxd\tt, \; \forall \sig \in \Ga \}\end{equation}
 We refer to the appendix for a short study of the space
$\Wd(\Om)$, in particular the definition of the application
$\gamma(V)$ as the
trace application.
\paragraph{} There
are three definitions of solutions : the classical (or
regular) solutions, the mild solutions (\cite{EngelNagel}
II.6, p.145) and the distributional
solutions (\ref{solutions_faibles} with the source term $f=0$), the
second and third ones
being two \textit{a priori} different types of weak
solutions.  The
following proposition proves
that the weak and mild solutions are the same ones.
\begin{prop}\label{solutions_mild_solutions_distrib}
Let $\rho \in \C([0,\infty[;L^1(\Om))$, then
$$ (\rho \text{ is a mild solution of
\eqref{equation_homogene}})\Leftrightarrow (\rho \text{ is a
weak solution of \eqref{equation_homogene}})$$
\end{prop}
\begin{proof}
$\bullet$\emph{First implication $\Rightarrow$} : It comes from the
fact that mild solutions are limit of classical ones which are weak
solutions in the sense of definition \ref{solutions_faibles}, by
passing to the limit in the identity \eqref{identite_distributions2}.
\\
$\bullet$\emph{Second implication $\Leftarrow$ :} Let $\rho \in
\C([0,\infty[;L^1(\Om))$ be a
weak solution in the sense of definition \ref{solutions_faibles} with
$f=0$. Define the function $R(t)=\int_0^t\rho(s)ds$.
We verify now
that $R(t) \in\Wd(\Om)$ by
using the definition. Fix $t\geq 0$ and a function $\phi \in
\C^1_c(\Om)$.
Using the function $\phi(t,x,\tt)\equiv \phi(x,\tt)$  in
\eqref{identite_distributions2}, we have
$$\int_\Om \int_0^t \rho(s) ds (G\cdot \nabla \phi)dxd\tt=-\int_\Om
\rho^0(x,\tt)
\phi(x,\tt)dxd\tt + \int_\Om \rho(t,x,\tt) \phi(x,\tt)dxd\tt$$
Therefore
 $R(t)  \in \Wd(\Om) \text{ and }
 \rho(t)=A\int_0^t \rho(s)ds + \rho^0$. \\
\espace  We now prove the boundary condition part contained
in order to have $R(t) \in D(A)$. Let $\phi(\sig)$ be a continuous
function
on $\Ga$, with compact support in $\Ga^*$. We can extend it to a
function of $\C_c(\barre{\Om}^*)$, still denoted by $\phi$, by
following the characteristics and truncating, namely
: $\phi(\Phi_\tau(\sig))=\phi(\sig)\zeta(\tau),\;  \tau
\in [0,+\infty[,\; \sig \in \Ga$
with $\zeta(\tau)$ being any regular function with compact support in
$[0,+\infty[$ such that $\zeta(0)=1$. Now, using the
density of $\C^1_c(\barre{\Om}^*)$ in $\C_c(\barre{\Om}^*)$, choose a
family $\phi_\eps \in \C^1_c(\barre{\Om}^*)$ such that $\phi_\eps
\xrightarrow{L^\infty} \phi$. For each $\eps$, using the remark
following the definition of weak solutions with the test
function $\phi_\eps(t,x,\tt)\equiv \phi_\eps(x,\tt)$, we have for
every $t \geq 0$
\begin{align*}
\int_\Om R(t) G\cdot \nabla \phi_\eps + \int_\Om
\rho^0(x,\tt)\phi_\eps(x,\tt)dxd\tt - & \int_\Om
\rho(t,x,\tt)\phi_\eps(t,x,\tt)dxd\tt = \\
&\int_\Ga N(\sig)\phi_\eps(\sig)d\sig \int_\Om \beta(x,\tt) R(t)
dxd\tt
\end{align*}
As $R(t) \in \Wd(\Om)$, and  $-\div(GR)=\rho - \rho^0$ by
passing to the limit in $\eps$, we obtain
$$\int_\Ga \gamma(R(t)) \Gn
\phi = \int_\Ga N\phi \int_\Om \beta R, \quad \forall t \geq 0$$
This identity being true for any function $\phi \in \C_c(\Ga^*)$, we
have the required boundary condition on $R(t)$. This ends the proof.
\end{proof}
\section{Properties of the operator}\label{properties_operator}
 We
first remark that  $(A,D(A))$ is closed, by classical considerations
and the
continuity of the trace application (prop.\ref{trace_IPP}).

 \subsection{Density of $D(A)$ in $L^1(\Om)$ and adjoint
$(A^*,D(A^*)$) of the operator}
 \begin{prop}\label{density}
The space $D(A)$ is dense in $L^1(\Om)$
\end{prop}

\begin{proof}
The proof follows the one done in \cite{BanksKappel} in dimension 1,
although some technical
difficulties appear in dimension 2. Since $\C^1_c(\Om)$ is dense in
$L^1(\Om)$, it is sufficient to
approximate any function $f\in
\C^1_c(\Om)$ by functions of $D(A)$, for the $L^1$ norm. Thus let
$f\in\C^1_c(\Om)$ be a fixed
function. Let
$\Sigma \subset \subset \Ga^*=\Ga\backslash (b,b)$ be the support of
$N(\sig)$ and for each $n\in
\N$
let $V_n$ be an open neighborhood of $\Sigma$ such that ${\rm
mes}(V_n) \rightarrow 0$, and $ (b,b)\notin \barre{V_n}$.
There exists a function $\phi_n \in \C^1_c(\R^2)$ such that
$$\phi_n(x,\tt) = \left\lbrace\begin{array}{cc} 1  &\text{ if
}(x,\tt)\in \Sigma \\
						0  &\text{ if
}(x,\tt)\in V_n^c
                             \end{array}\right.
\quad 0\leq \phi_n \leq 1$$
Then, we extend the function $H(\sig)=\frac{N(\sig)}{-G\cdot \nuu
(\sig)}\; : \; \Ga^* \rightarrow
\R$ to a Lipschitz function $\barre{H} : \barre{V_n} \cap \barre{\Om}
\rightarrow \R$ (for example by following the characteristics).
Let

$$h_n(x,\tt)=\left\lbrace\begin{array}{ll} (\barre{H}\phi_n)(x,\tt) &
\text{ if }(x,\tt) \in
\barre{V_n}\cap \barre{\Om} \\
					   0			    &
\text{ if }(x,\tt)
\barre{\Om}\backslash\barre{V_n}
                        \end{array}\right.$$
 It satisfies $h_n \in W^{1,\infty}(\Om)$ and  $h_n
\xrightarrow{L^1} 0$. Let
$f_n = f + a_n h_n$,
with
$$a_n=\frac{\int_\Om \beta f dxd\tt}{1-\int_\Om \beta h_ndxd\tt}.$$
Since $||h_n||_{L^1(\Om)} \rightarrow 0$ and
$\beta$ is in $L^\infty$, for $n$ sufficiently large $1 - |\int_\Om
\beta h_n dxd\tt| \geq 1/2$
and $|a_n|\leq 2 ||\beta||_{L^{\infty}}||f||_{L^1}$.
Then  $f_n
\xrightarrow{L^1} f $ and furthermore, since $h_n \in
W^{1,\infty}(\Om)\subset W^{1,1}(\Om) \subset \Wd(\Om)$, we have $f_n
\in D(A)$.
\end{proof}
We are now interested in characterizing the
adjoint of the
operator $(A,D(A))$. We will see that the first eigenvector of
$(A^*,D(A^*))$ plays
an important role in the structure of the equation  in the asymptotic
behavior (see theorem \ref{thm_princ}).
\begin{prop}[Domain and expression of $A^*$]\label{expression_A*}
 \begin{equation}\label{D(A*)}D(A^*)=\{U \in L^\infty;
\; G\cdot\nabla U\in
L^\infty\}:=\Wd^\infty(\Om)\end{equation}
$$A^*U=G\cdot\nabla U + \beta \int_\Ga
U(\sig)N(\sig)d\sig .$$
\end{prop}
\begin{proof}
$\bullet$ The first inclusion for the domain of $A^*$ is a
consequence of the property \ref{trace_IPP}. The second inclusion
$D(A^*)\subset \Wd^\infty(\Om)$ requires a little much of
work. For a function $U
\in D(A^*)$, we will show that $\phi \mapsto <U,\div(G\phi)>$ can be
extended in a continuous linear form on
$L^1$, which will allow us to conclude using the Riesz theorem that
$U\in \Wd^\infty$. To do this,
it is sufficient to show
that there exists a constant $c\geq 0$ such that
\begin{equation}\label{forme_lineaire}
|< U,A\phi>_{\D',\D}| \leq c||\phi||_{L^1} \quad \forall \phi \in
\D(\Om)
\end{equation}
 This is almost
done by the definition of the domain $D(A^*)$ except the fact that
$\D(\Om)$ is not a subset of $D(A)$. We are driven to use the
following trick. Define the space :
$$\D_0(\Om)=\{\phi \in \D(\Om); \; \int_\Om \be \phi =0\}$$
which is a subspace of $D(A)$. We will project a given
function in $\D(\Om)$ on $\D_0(\Om)$. Let $\phi_1 \in \D(\Om)$
be a fixed function such that $\ds \int_\Om \beta \phi_1 =1$. Then
$$\phi =
\underset{\in \D_0(\Om)\subset D(A)}{\underbrace{\phi - \left(\int \be
\phi \right) \phi_1}} +
\underset{\in \R \phi_1}{\underbrace{\left(\int \be \phi \right)
\phi_1}}.$$
So eventually, denoting as $c_1$ the constant given by the
belonging of $U$ to $D(A^*)$
\begin{align*}|<U,A\phi>_{\D',\D}| & =|<U,A(\phi - \left(\int \be \phi
\right)
\phi_1)>+<U,\left(\int \be \phi \right) A\phi_1>| \\
                                  & \leq (c_1 + c_1
||\be||_{L^\infty}||\phi_1||_{L^1} +
||\be||_{L^\infty} ||U||_{L^\infty}||A\phi_1||_{L^1}) ||\phi||_{L^1}
\end{align*}
which shows \eqref{forme_lineaire}  and thus yields
the result.
\end{proof}

 \subsection{Spectral properties and dissipativity}
In order to have a candidate for a stable asymptotic distribution of
the solutions of
our equation, we are interested in the stationary eigenvalue problem :
 \begin{eqnarray}\label{eqn_eigenproblem}
\left\{ \begin{array}{c}
(\lam,V,\Psi)\in \R_+^* \times D(A) \times D(A^*) \\
AV=\lam V,\quad A^*\Psi = \lam \Psi \\
\quad \int_\Om V \Psi dxd\tt =1,\; \Psi\geq 0,\; \int_\Ga N\Psi
d\sig=1
\end{array} \right.
\end{eqnarray}
\begin{prop}\label{eigenproblem}[Existence of solutions to the
eigenproblem]
Under the assumption
\begin{equation}\label{spectral_condition}
\ds \int_0^\infty \int_\Ga \be(\Phi_\tau(\sig))N(\sig)d\tau d\sig>1,
\end{equation}
there exists a unique solution $(\lam_0,V,\Psi)$ to the eigenproblem
\eqref{eqn_eigenproblem}. Moreover,
we have the
following \textbf{spectral equation} on $\lam_0$ :
\begin{equation}
\boxed{\int_0^{+\infty}\int_{\Ga}\be(\Phi_\tau(\sig))N(\sig)
e^{-\lam_0 \tau}d\tau d\sig =1}
\end{equation}
The direct eienvector is given by
$$V(\Phi_\tau(\sig))=C_{\lam_0}N(\sig)e^{-\lam_0
\tau}|J_\Phi|^{-1}, \quad \forall \tau >0,\; a.e
\; \sig
\in \Ga$$
where $C_{\lam_0}$ is a positive constant and $|J_\Phi|$ is the
jacobian of $\Phi$ from section \ref{model_ODE}. The adjoint
eigenvector $\Psi$ is
given by :
\begin{equation}\label{expression_Psi}
\Psi(\Phi_\tau(\sig))= e^{\lam_0 \tau}\int_\tau^\infty
\be(\Phi_s(\sig)) e^{-\lam_0 s} ds
\quad \forall \tau >0, \; a.e \; \sig
\in \Ga .
\end{equation}
Hence we have
$$\frac{\inf \beta}{\lam_0} \leq  \Psi(x,\tt) \leq
\frac{\sup\beta}{\lam_0} \quad \forall (x,\tt)\in \Om$$
\end{prop}

\begin{rem}
In the model we use in practice, where $\beta(x,\tt)=m x^\alpha$ the
condition \eqref{spectral_condition} is fulfilled since $\ds
\int_0^\infty \int_\Ga \be(\Phi_\tau(\sig))N(\sig)d\tau d\sig
=\infty$, and the inequalities on $\Psi$ write
$$\frac{m}{\lam_0} \leq  \Psi(x,\tt) \leq
\frac{m b^\alpha}{\lam_0} \quad \forall (x,\tt)\in \Om$$
\end{rem}

\begin{proof}
$\bullet$ \textit{The direct eigenproblem.}\\
$\diamond$ We use the following change of variable, which
consists in transforming a function of
$\Wd(\Om)$ into a function of $\W$ :
$$\tilda{V}(\tau,\sig)=-V(\Phi_\tau(\sig))|J_\Phi| , \quad \forall
\tau \in [0,+\infty[,\;\sig \in \Ga$$
where we recall that $|J_\Phi|=-\Gn(\sig)e^{\int_0^\tau
\div(G(\Phi_s(\sig)))ds}$ is the jacobian of the
application $\Phi :(\tau,\sig) \mapsto \Phi_\tau(\sig)$ (see
section \ref{model_ODE}).\\
Rewriting the problem on $\V$ and denoting
$\tilda{\beta}(\tau,\sig)=\beta(\Phi_\tau(\sig))$, we
get
\begin{eqnarray}\label{pb_spectral_chgtVar}
\left\lbrace\begin{array}{c}
             \d_\tau \V+\lam \V =0 \\
	      \V(0,\sig)=N(\sig)\int \tilda{\beta} \V d\tau d\sig'
            \end{array}
\right.
\end{eqnarray}
Direct computations show that Problem
\ref{pb_spectral_chgtVar} has a solution if
\begin{equation}\label{eqnSpectrale}
1=\int_0^\infty\int_\Ga N(\sig)\tbeta(\tau,\sig)e^{-\lam \tau}  d\sig
d\tau
\end{equation}
 and conversely, if $\lam_0$ is a solution of the
equation \eqref{eqnSpectrale}, we get solutions to the problem
\eqref{pb_spectral_chgtVar} given by
\begin{equation}\label{Vtilde}\V(\tau,\sig)=C_{\lam_0} N(\sig)
e^{-\lam_0 \tau}\end{equation}
and we can then fix the constant $C_{\lam_0} >0$ in order
to have the normalization condition $1=\int_\Om V \Psi dx d\tt$
with $\Psi$ the dual eigenvector defined below.\\
$\diamond$ We now prove that there exists a unique solution to
equation \eqref{eqnSpectrale} under the hypothesis
\eqref{spectral_condition}. Indeed, let us define the function $F:\R
\rightarrow \barre{\R}$  by
\begin{equation}\label{definition_F}
F(\lam)=\int_0^\infty\left(\int_\Ga
N(\sig)\tbeta(\tau,\sig)\right)e^{-\lam \tau} d\sig d\tau
\end{equation}
It is the Laplace transform of the function $\tau \mapsto \int_\Ga
N(\sig)\tbeta(\tau,\sig)d\sig$.
The condition \eqref{spectral_condition} means that $F(0)>1$ and $F$
being strictly decreasing on
$\R$
and continue on $]0,+\infty[$, the equation \eqref{eqnSpectrale} has a
unique solution in $\R$,
$\lam_0 \in ]0,+\infty[$.

$\diamond$ From \eqref{Vtilde}, we obtain that
 $\V \in \W$.  Using theorem
\ref{conjugaisonWdiv} from the appendix, we  deduce that  $V \in
\Wd(\Om)$.
\begin{rem}
Here the theorem \ref{conjugaisonWdiv} takes its interest since it is
not completely obvious
that the composition of $\V$ by $\Phi^{-1}$ would give a function in
$\Wd(\Om)$, due to the fact
that the change of variable $\Phi$ (and $\Phi^{-1}$) is not globally
Lipschitz.
\end{rem}
\noindent$\bullet$ \textit{The adjoint eigenproblem.} \\
$\diamond$ Expression of $\Psi$. Using the expression of the adjoint
operator $(A^*,D(A^*))$ from the proposition  \ref{expression_A*}, the
adjoint spectral problem reads, along the characteristics : find
$\Psi \in \Wd^\infty(\Om)$ such that
\begin{equation}\label{equation_psi_caract}\d_\tau\Psi
(\Phi_\tau(\sig))=\lam_0
\Psi(\Phi_\tau(\sig)) - \be(\Phi_\tau(\sig)) \int_\Ga
\Psi(\sig')N(\sig')d\sig'
\end{equation}
from which we get, for each function $\Psi(\sig)$ defined on the
boundary, a solution to the
equation given by
\begin{equation}\label{expression_Psi1}
\Psi(\Phi_\tau(\sig))=\Psi(\sig)e^{\lam_0 \tau} - \int_\Ga
\Psi(\sig')N(\sig')d\sig' \int_0^\tau\be(\Phi_s(\sig))e^{\lam_0
(\tau-s)}
\end{equation}
$\diamond$ Non-negative solution. To get a non-negative
solution we are driven to the
following condition
$$\Psi(\sig) \geq \int_\Ga \Psi(\sig')N(\sig')d\sig' \int_0^\infty
\be(\Phi_s(\sig)) e^{-\lam_0 s}ds, \quad a.e\;\sig \in \Ga$$
Now, if the inequality is strict, multiplying by $N(\sig)$ and
integrating on $\Ga$ gives $1 > \int_\Ga \int_0^\infty
\be(\Phi_s(\sig)) e^{-\lam_0 s}dsd\sig$ which belies the spectral
equation \eqref{eqnSpectrale}.
We are thus driven to choose
\begin{equation}\label{condition}\Psi(\sig)=\int_\Ga
\Psi(\sig')N(\sig')d\sig'
\int_0^\infty\be(\Phi_s(\sig))e^{-\lam_0 s}ds, \quad \forall \sig \in
\Ga
\end{equation}
Defining $g(\sig)=\int_0^\infty\be(\Phi_s(\sig))e^{-\lam_0 s}ds$, this
means that $\Psi(\sig)$ is
in the vector space generated by $g$. Then it remains to have the
suitable normalization constant.
Remembering  the spectral equation \eqref{eqnSpectrale} verified by
$\lam_0$ shows that the function $\Psi(\sig)=g(\sig)$ is appropriate.
We finally get \eqref{expression_Psi} from \eqref{expression_Psi1},
which gives $\Psi \in L^\infty$ and $||\Psi||_{L^\infty} \leq \frac
{||\be||_{L^\infty}}{\lam_0}$. \\
$\diamond$ Regularity of $\Psi$. Using the equation
\eqref{equation_psi_caract} verified
by $\Psi$ we get $\d_\tau \Psi(\Phi_\tau(\sig))\in L^\infty$ and so
using the conjugation theorem of $\Wd^\infty(\Om)$ and
$W^{1,\infty}((0,\infty)\,;\,L^\infty(\Ga))$
(theorem \ref{conjugaisonWdiv}), we have $\Psi \in
\Wd^\infty(\Om)$.
\end{proof}
Using the change of variables
$\V(\tau,\sig)=V(\Phi_\tau(\sig))|J_\Phi|$, the theorem
\ref{conjugaisonWdiv} and the proposition \ref{trace_IPP}, we can
follow the methods of the one-dimensional case done in
\cite{BBHV,BanksKappel} thanks to the decoupling of
$\beta(x,\tt,\sig)$ in $N(\sig) \times\beta(x,\tt)$, to obtain the
following proposition.
 \begin{prop}
(i) For $Re(\lam) > \lam_0$, we have $Im(\lam I - A)=L^1(\Om)$. \\
(ii) The operator $(A-\om I,D(A))$ is dissipative for every $\om
\geq ||\be||_{L^\infty(\Om)}$.
\end{prop}
Applying the Lumer-Philips theorem, we obtain
\begin{cor}\label{prop_semigroupe}
Under the assumptions \eqref{hypotheses_donnees} the operator
$(A,D(A))$ generates a semigroup on
$L^1(\Om)$ denoted by $e^{tA}$ and we have
$$|||e^{tA}|||\leq e^{t||\beta||_{L^\infty}}$$
\end{cor}


\section{Existence and asymptotic behavior}\label{existence_asymp}
 \subsection{Well-posedness of the equation}
   \subsubsection{Existence for the non-homogeneous problem}

\begin{prop}\label{prop_existence_nonhomogene}~\\
$(i)$ Let $f\in L^1(]0,\infty[;L^1(\Ga))$ and assume
\eqref{hypotheses_donnees}. There exists a unique
solution of the non-homogeneous problem \eqref{equation_nonhomogene},
denoted by $\T f$ and we have
$$\T f \in \C([0,\infty[;L^1(\Om))$$
$(ii)$ If $f\in \C^1([0,\infty[;L^1(\Ga))$ and $f(0)=0$ then
$$\T f \in \C^1([0,\infty[;L^1(\Om))\cap \C([0,\infty[;\Wd(\Om))$$
Moreover, we have the positivity property
$$(f\geq 0) \Rightarrow (\T f \geq 0)$$
\end{prop}

\begin{proof}
The proof is based on a fixed point argument. It is divided in three
steps : first we prove the point (ii) using the Banach fixed point
theorem, then thanks to an estimate in $\C([0,\infty[,L^1(\Om))$ we
construct the weak solutions as limits of regular solutions, and
finally we prove uniqueness. \\
$\bullet$ \textit{Step 1}. $\diamond$ As usual now, we first
simplify the problem using the conjugation theorem (theorem
\ref{conjugaisonWdiv}). We use the change of variable
$\trho=\rho(\Phi_\tau(\sig)) |J_\Phi|$ and still denoting $\rho$ for
$\trho$ and $\beta$ for $\tilda{\beta}=\beta(\Phi_\tau(\sig))$, we
consider the following non-homogeneous problem with nonzero initial
condition
\begin{eqnarray}\label{equation_nonhomogene_CI}\left\lbrace
\begin{array}{c}
\d_t \rho + \d_\tau \rho =0 \\
\rho(t,\sig)=N(\sig)\int \beta w + f(t,\sig) \\
\rho(0)=\rho^0
\end{array}\right.
\end{eqnarray}
Let $\rho^0 \in D(A)$ and $f \in \C^1([0,\infty[;L^1(\Ga))$ with
$f(0)=0$. For $T>0$ we define the space
$$X_T=\{w \in \C^1([0,T];L^1(]0,\infty[; L^1(\Ga));\;
w(0,\cdot)=\rho^0 \}$$
It is a complete metric space. To $w\in X_T$ we associate the solution
$\rho$ of the equation \eqref{equation_nonhomogene_CI}, namely
$$\rho(t,\tau,\sig)=\underset{:=H(t-\tau,\sig)}{\underbrace{
\left\lbrace N(\sig)\int_0^\infty \int_\Ga \beta
w(t-\tau,\tau',\sig')d\tau' d\sig' +f(t-\tau,\sig) \right\rbrace}}
\mathbf{1}_{t >
\tau} + \rho^0(\tau - t,\sig) \mathbf{1}_{t<\tau}$$
and define the linear operator $\T_{\rho^0,f}$ by $\T_{\rho^0,f}
w:=\rho$. Note here that $w \geq 0$ implies $\rho \geq 0$ if
$\rho^0 \geq 0$ and $f \geq 0$, and that $H \in
\C^1([0,T];L^1(\Ga))$.\\
$\diamond$ Regularity of $\rho$. We now show that $\rho \in
X_T$ and that $\rho \in \C([0,T];\W)$. Indeed we have
\begin{equation}\label{eqn1}
\rho(t,\tau,\sig)\mathbf{1}_{t>\tau} =
H(t-\tau,\sig)\mathbf{1}_{t>\tau}, \quad
\rho(t,\tau,\sig)\mathbf{1}_{t<\tau}=\rho^0(\tau - t,\sig)
\mathbf{1}_{t<\tau} .
\end{equation}
From these expressions we get that $\rho \in
\C([0,T];L^1(]0,\infty[;L^1(\Ga)))$ since the two functions $H$ and
$\rho^0$ are in
$L^1$. \\
\espace Moreover, $H(0,\sig)=N(\sig)\int \beta \rho^0 = \rho^0(0)$
from the compatibility conditions contained in the facts that $w\in
X_T$, $f(0)=0$ and $\rho^0 \in D(A)$. This allows to conclude that
$\rho(t, \cdotp) \in \C([0,\infty[;L^1(\Ga))$.
Furthermore, from the expressions \eqref{eqn1}, we
see that for each $t$ $\rho(t,\cdot) \in W^{1,1}((0,t),L^1(\Ga))\cap
W^{1,1}((t,\infty),L^1(\Ga))$ since $\rho^0 \in D(A)$ and $H \in
\C^1([0,T];L^1(\Ga))$.
Combining this with the continuity in $\tau$ gives $\rho(t,\cdot) \in
\W$. Finally from the expression of $\d_\tau \rho$ obtained
differentiating in $\tau$ the expressions \eqref{eqn1} we get
$\rho \in \C([0,T],\W)$. \\
\espace It remains to show that $\rho \in
\C^1([0,T];L^1(]0,\infty[;L^1(\Ga)))$. For the sake of
simplicity we
forget the dependency on $\sig$. We define for almost every $t$ and
$\tau$
$$\d_t \rho(t,\tau) := \d_t H(t-\tau) \mathbf{1}_{t>\tau} -  \d_t
\rho^0(\tau -t)\mathbf{1}_{t<\tau}$$
Now we compute
\begin{align*}
\frac{1}{h}||&\rho(t+h) - \rho(t) - h\d_t
\rho(t)||_{L^1([0,\infty[)}  = \\
& \frac{1}{h}||H(t+h -\cdot) -
H(t-\cdot) - h\d_t H(t-\cdot)||_{L^1([0,t[)} \\
			 &+\underset{A}{\underbrace{\frac{1}{h}||H(t+h
- \cdot) - \rho^0(\cdot
-t) - h \d_t \rho^0(\tau - t)||_{L^1(]t,t+h[)} }} \\
			 &+ \frac{1}{h}||\rho^0(\cdot - t -h) -
\rho^0(\cdot-t) - h\d_t
\rho^0(\cdot - t)||_{L^1([t+h,\infty[)}
\end{align*}
The first and the last terms go to zero when $h$ tends to zero since
$H$ is in $\C^1([0,T];L^1(]0,\infty[))$ and $\rho^0$ is in $D(A)$. To
deal with
the last term $A$, we write
$$A\leq \frac{1}{h}\int_t^{t+h}|H(t+h-\tau) - \rho^0(\tau-t)|d\tau +
\int_t^{t+h}\d_t\rho^0(\tau-t)d\tau$$
The first term goes to zero because of the compatibility condition
$H(0)=\rho^0(0)$ and also the last one because $\d_t\rho^0 \in L^1$.
We can then conclude $\rho \in \C^1([0,T];L^1(]0,\infty[))$. \\
$\diamond$ The previous considerations show that the operator
$\T_{\rho^0,f}$ has values in $X_T$. Now, if $w_1$ and $w_2$ are in
$X_T$ we compute
$$||\T_{\rho^0,f} w_1 - \T_{\rho^0,f} w_2 ||_{X_T} \leq T
||\beta||_{L^\infty}  ||w_1 -w_2||_{X_T}$$
Using a bootstrap argument we prove the existence of a
solution on $[0,\infty[$ and transporting the regularity facts back to
$\Om$ by using the conjugation theorem \ref{conjugaisonWdiv} ends the
point $(ii)$. \\
$\bullet$ \textit{Step 2}. Denote by $\T f$ the fixed
point
of the operator $\T_{0,f}$, defined up to now only when $f$ is regular
and satisfies the compatibility condition $f(0)=0$, one
has
\begin{lem}
Let $f\in \C^1([0,\infty[;L^1(\Ga))$ with $f(0)=0$ and $\T f$ be the
solution of the
equation $\eqref{equation_nonhomogene_CI}$ with a zero initial
condition. Then for all $T>0$
$$||\T f||_{\C([0,T];L^1(\Om))} \leq e^{T ||\beta||_{\infty}} \int_0^T
|f(s)|e^{-||\beta||_{\infty} s}ds$$
\end{lem}
\begin{proof}
The solution $\T f=\rho$ being regular, the function $|\rho|$ also
verifies the equation and integrating on $\Om$ yields
\begin{align*}
\frac{d}{dt}\int_\Om |\rho|(t)d\tau d\sig & = |\int_\Om \beta\rho(t)
dx
d\tt+ \int_\Ga f(t,\sig) d\sig| \leq
||\beta||_{\infty}\int_\Om |\rho|(t)dx d\tt + \int_\Ga|f(t,\sig)|d\sig
\end{align*}
and a Gronwall lemma gives the result.
\end{proof}
Now by a density argument and using the previous lemma, we can
construct a solution $\T f \in \C([0,\infty[;L^1(\Om))$ when $f \in
L^1(]0,\infty[;L^1(\Ga))$. This solution can be constructed
non-negative whenever $f$ is non-negative itself. \\
$\bullet$ \textit{Step 3}. It remains to show the uniqueness of
the solution. If $\rho_1$ and $\rho_2$ are two solutions of the
non-homogeneous equation \eqref{equation_nonhomogene_CI}, then $\rho_1
- \rho_2$ is a weak solution of the homogeneous equation
\eqref{equation_homogene} with zero initial condition. From the
proposition \ref{solutions_mild_solutions_distrib} the weak solutions
in the sense of the distributions are the same than the mild solutions
and thus $\rho_1 - \rho_2$ is a mild solution of the
homogeneous
equation and hence is zero by uniqueness of the mild
solutions\end{proof}

   \subsubsection{Existence for the global problem}
\begin{thm}[Existence and
uniqueness]\label{thm_princ}~\\
$\bullet$ Let $\rho^0 \in L^1(\Om)$ and $f\in
L^1(]0,\infty[\times\Ga)$, and assume \eqref{hypotheses_donnees}.
There exists a unique weak solution of the
equation \eqref{equation}, given by
$$\rho= e^{tA}\rho^0 + \mathcal{T}f$$
with $\mathcal{T}f$ being a weak solution of the non-homogeneous
equation \eqref{equation_nonhomogene}. \\
$\bullet$ If $\rho^0 \in D(A)$ and $f\in \C^1([0,\infty[;L^1(\Ga))$
and verifies $f(0)=0$, then we have
$$\rho \in \C^1([0,\infty[;L^1(\Om))\cap \C([0,\infty[;\Wd(\Om))$$
\end{thm}

 \subsection{Properties of the solutions and asymptotic behavior}

In the next proposition, we prove some useful properties
of the solutions, which appear in the $L^1_\Psi$ norm  defined by
\begin{equation}\label{L1psi}||f||_{L^1_\Psi}=\int_\Om |f| \Psi
dxd\tt,\end{equation}
with $\Psi$ the dual eigenvector from proposition \ref{eigenproblem}.
We should notice that when $\beta
\in L^\infty$ and $\beta \geq \delta >0$, by the inequalities from
proposition \ref{eigenproblem}, the $L^1_\Psi$ norm is equivalent to
the $L^1$ norm. Hence the solutions have finite $L^1_\Psi$ norm. The
main idea in the proof of the following proposition is to use various
entropies in the
space $L^1_\Psi$, and is based on ideas from \cite{perthame} and
\cite{perthameMichel}

\begin{prop}\label{properties_solutions}
Let $\rho^0 \in L^1(\Om)$ and $\rho$ the solution of the equation
\eqref{equation}. The
following properties hold :
\begin{enumerate}
\item[(i)]
\begin{equation}\label{contraction}
\int_\Om |\rho(t)|\Psi \leq e^{\lam_0 t}\left\lbrace \int_\Om
|\rho^0|\Psi + \int_0^t\int_\Ga
\Psi(\sig)e^{-\lam_0 s}|f|(s,\sig)d\sig ds\right\rbrace,\quad \forall
t\geq 0
\end{equation}
\item[(ii)](Evolution of the mean-value in $L^1_\Psi$)
$$\int_\Om \rho(t)\Psi = e^{\lam_0 t}\left\lbrace \int_\Om \rho^0\Psi
+ \int_0^t\int_\Ga\Psi(\sig)e^{-\lam_0 s}f(s,\sig)d\sig
ds\right\rbrace,\quad \forall t\geq 0$$
\item[(iii)](Comparison principle) If $f\geq 0$
$$\rho^0_1\leq \rho^0_2 \quad \Rightarrow \quad \rho_1 (t) \leq \rho_2
(t) \quad \forall t\geq 0$$
\end{enumerate}
\end{prop}

\begin{proof}
Each time we aim to prove something on weak solutions, we will start
proving it for classical
solutions and then use the density of $D(A)$ to conclude. So, let do
the calculations with a strong
solution $\rho$ associated to an initial condition $\rho^0$ in $D(A)$
and a function $f \in \C^1(]0,\infty[;L^1(\Ga))$ with $f(0)=0$,
for which the calculations can
be justified. We first remark that the dual eigenvector $\Psi$ which
belongs to $\Wd^\infty(\Om)$
verifies the following equation :
\begin{equation}\label{eq_psi} G\cdot \nabla \Psi - \lam_0 \Psi = -\be
\end{equation}
since by the construction of $\Psi$ and the spectral equation
$\int_\Ga \Psi(\sig) N(\sig)d\sig=1$.
Defining $\tilda{\rho}(t,x,\tt)=e^{-\lam_0 t}\rho(t,x,\tt)$
we have the following equation on $\tilda{\rho}$ :
\begin{equation}\label{equation_lam0}
\d_t \trho + \div(G \trho) + \lam_0 \trho = 0,
\end{equation}
with the same initial condition as for $\rho$ and a suitable boundary
condition. Using that $\trho \in
\Wd(\Om)$, $\Psi \in \Wd^\infty(\Om)$ and the proposition
\ref{Wdiv_produit_Lipschitz}, we obtain
the following equation on $\trho \Psi$ :
\begin{equation}\label{eqn_rhoPsi}
\d_t (\trho \Psi) + \div(G \trho \Psi) = -\be \trho
\end{equation}
$(i)$ Let us first state the following lemma.
\begin{lem}\label{lem_eqn_val_abs}
Let $\rho^0 \in L^1(\Om)$ and $\rho$ be the associated weak solution
of the equation \eqref{equation}. Then the function  $|\rho|$ solves
the same equation, with suitable initial and boundary conditions.
\end{lem}
\begin{proof}
For a regular solution of the equation $\rho$ associated to
a regular initial condition $\rho^0 \in
D(A)$ and a regular data $f$, we can use the proposition
\ref{Wdiv_produit_Lipschitz} with
the function $H(\cdot)=|\cdot|$ to have that
$|\rho(t)| \in \Wd(\Om)$ and
$$\div(G|\rho|)= \sgn(\rho) G\cdot \nabla \rho + |\rho| \div(G)$$
Since $\rho$ is regular in time, by multiplying the equation by
$\sgn(\rho)$ we get the result. For
a solution $\rho(t) \in L^1(\Om)$ we obtain the result by density of
the strong solutions.
\end{proof}
Thanks to this lemma we have the equation \eqref{eqn_rhoPsi} written
on $|\trho|$, from which we get, integrating in $(x,\tt)$, that
\begin{align*}
\frac{d}{dt}&\int_\Om |\trho| \Psi dxd\tt  =- \int_\Ga
\gamma(|\trho|)\Psi G\cdot\nuu d\sig - \int_\Om
\be(x,\tt)|\trho(t,x,\tt)|dxd\tt \\
                                    & =
\int_\Ga \Psi(\sig)\left|
N(\sig)\int_\Om\be(x,\tt)\trho(t,x,\tt)dxd\tt + e^{-\lam_0 t}f(t,\sig)
\right|  - \int_\Om \be(x,\tt)|\trho|(t,x,\tt)|dxd\tt \\
				     & \leq \int_\Ga |f(t,\sig)|
\Psi(\sig)
\end{align*}
and thus deduce the first property by integrating in time. To
deal with weak solutions we again use the density of regular
solutions.\\
$(ii)$ To obtain the evolution of the mean value, we
integrate in space and use again a density argument.\\
$(iii)$ Writing the solution of the global problem as $\rho
=e^{tA}\rho^0 + \T f$, we only have to prove the positivity
for the homogeneous part since the positivity of the
non-homogeneous one has been established in the proposition
\ref{prop_existence_nonhomogene}. It can be
proved in the same
way as the first point but using the negative part function
instead of the absolute value.  \end{proof}

\begin{prop}[Asymptotic behavior]\label{asymp_behav}
Assume that
$$\ds \int_0^\infty \int_\Ga \be(\Phi_\tau(\sig))N(\sig)d\tau
d\sig>1,$$
and that there exists $\mu >0$ such that
$\beta - \mu \Psi \geq 0.$
Let $\rho^0\in L^1(\Om)$, $f\in L^1(]0,\infty[\times \Ga)$, $\rho$
the associated solution to the global problem and
$(\lam_0,V,\Psi)\in \R^*_+ \times D(A) \times D(A^*)$ be solutions to
the direct and adjoint eigenproblems. We have :
\begin{align*}\boxed{||\rho(t) e^{-\lam_0 t} - m(t) V||_{L^1_\Psi}
\leq e^{-\mu
t}\left\lbrace  ||\rho^0 - m_0 V||_{L^1_\Psi}
  + 2\int_0^t e^{-(\lam_0-\mu)s }\int_\Ga |f|(s,\sig)\Psi(\sig) ds
\right\rbrace},
\end{align*}
where $||f||_{L^1_\Psi}=\ds \int_\Om |f| \Psi$, and $m(t)=\ds \int_\Om
\rho(t) \Psi=\ds \int_\Om \rho^0 \Psi + \int_0^t e^{-\lam_0 s
}\int_\Ga f(s,\sig)\Psi(\sig)d\sig ds$.
\end{prop}

\begin{rem}
Notice that choosing $\mu <\lam_0$ gives the convergence of the
integral $\int_0^\infty e^{-(\lam_0-\mu)s }\int_\Ga
|f|(s,\sig)\Psi(\sig)
ds$ and thus the convergence to zero of the right hand side of the
inequality.
\end{rem}

\begin{rem}
The hypothesis of the theorem are fulfilled in the case of
biological applications where $\beta(x,\tt)=m
x^{\alpha}$, because we have then $\beta \geq m >0$ and $\Psi \in
L^\infty$.
\end{rem}

\begin{proof}
Again we start with a regular solution $\rho(t,x,\tt)$. We then
follow the calculation done in \cite{perthame} III.7 pp.66-67,
adapting the method to take into account the contribution of the
source term. Define the function
$$h(t,x,\tt)=\rho(t,x,\tt) e^{-\lam_0 t} - m(t) V$$
which satisfies $\int_\Om h(t) \Psi =0$ for all non-negative $t$,
by the property of evolution of the mean value and since $ \int_\Om
V \Psi=1$. As the direct eigenvector $V$ solves the equation
$\eqref{equation_lam0}$, $h$ solves the equation
$$\d_t h + \div(hG) + \lam_0 h = -e^{-\lam_0 t} F V$$
where $F(t):=\ds\int_\Ga f(t,\sig)\Psi(\sig)$. Multiplying the
equation by the function $\sgn(h)$ gives the following equation on
$|h|$
$$\d_t |h| + \div(|h|G) + \lam_0 |h| = -e^{-\lam_0 t} F V \sgn(h)$$
Multiplying this equation by $\Psi$, the equation on
$\Psi$ by $|h|$ and then summing the both gives
$$\d_t (|h| \Psi) + \div(G |h| \Psi) = -\beta |h|- e^{-\lam_0 t}F
V \Psi \sgn(h)$$
Now integrating in $(x,\tt)$ yields :
\begin{align*}
\frac{d}{dt}\int_\Om |h| \Psi dxd\tt & =
\int_\Ga \Psi(\sig)\left|N(\sig)\int_\Om
\be h dxd\tt +  e^{-\lam_0 t}f(t,\sig)\right|d\sig
- \int_\Om \be|h|dxd\tt \\
&  - e^{-\lam_0 t}F\int_\Om V \Psi\sgn(h)dxd\tt\\
\end{align*}
\vskip-0.9cm\noindent Now we use that
\begin{align*}
\left|N(\sig)\int_\Om \be h dxd\tt +  e^{-\lam_0
t}f(t,\sig)\right|= & \left(N(\sig)\int_\Om \be h
dxd\tt\right) \sgn(h(\sig)) + \\
& \left( e^{-\lam_0 t}f(t,\sig)\right) \sgn(h(\sig))
\end{align*}
to obtain
\begin{align*}
\frac{d}{dt}\int_\Om |h| \Psi dxd\tt = &
\underset{A}{\underbrace{\int_\Ga
\Psi(\sig)N(\sig)\sgn(h(\sig)) d\sig \left| \int_\Om \be h
dxd\tt\right| - \int_\Om \be|h|dxd\tt}} \\
&+ \underset{B}{\underbrace{\int_\Ga e^{-\lam_0
t}f(t,\sig) \sgn(h(\sig)) \Psi(\sig)d\sig - e^{-\lam_0
t}F\int_\Om V \Psi\sgn(h)dxd\tt}}
\end{align*}
$\bullet$ We first deal with the term $A$. Using that $ \int_\Om h
\Psi=0$ and remembering that $\int_\Ga N \Psi = 1$ we compute
\begin{align*}
A                                    & \leq \left|\int_\Om
\be h dxd\tt - \mu\int_\Om
\Psi h dxd\tt \right| - \int_\Om
\be |h| dxd\tt \leq \int_\Om (\be-\mu
\Psi)|h|dxd\tt - \int_\Om
\be |h|dxd\tt \\
                                    & \leq -\mu \int_\Om \Psi |h|
dx d\tt
\end{align*}
where we used that $\beta - \mu\Psi \geq 0$. \\
$\bullet$ A direct majoration, the
positivity of the eigenvectors $V$ and $\Psi$ and the
fact that $\ds\int_\Om V\Psi =1$ gives, denoting
$\barre{F}(t):=\ds\int_\Ga |f(t,\sig)|\Psi(\sig)$, that $B\leq 2
e^{-\lam_0 t}\barre{F}(t)$ \\
$\bullet$ A Gronwall lemma finally gives
$$\int_\Om |h(t)|\Psi \leq e^{-\mu t}\left\lbrace \int_\Om
|h(0)|\Psi + 2\int_0^t e^{-(\lam_0-\mu)s}F(s)ds\right\rbrace$$
which is the required result. For an initial data in $L^1(\Om)$,
remark that it is possible to pass
to the limit in the previous expression.
\end{proof}

\section{Conclusion and perspectives}

In the present paper, we achieved the first step of our
program consisting in elaborating and applying a model of metastatic
growth including the tumoral angiogenesis process : the mathematical
analysis of the direct problem. To do this, we used semigroup
techniques and also the characteristics in order to study
the natural regularity of the solutions to our equation, which led us
to a short study of the space $\Wd(\Om)$. This theoretical study
brings to light the quantity $\lam_0$ as characterizing the asymptotic
growth of the metastatic process. This parameter has biological
relevance and finding the best way of controlling its value by means
of antiangiogenic drugs can be of great interest. The crucial problem
is now the identification of the parameters of the model from
biological data, in order to predict the optimized administration
protocol for antiangiogenic drugs.
\paragraph{} To achieve this, we need to perform efficient numerical
simulations of the equation. Due to the large disproportion of the
boundary condition and the solution itself, as well as the size of the
domain (typically $b=10^{11}$ for humans) and the behavior of the
characteristics
attached to the velocity field $G$ (see figure
\ref{caracteristiques}), performing good simulations of the
equation is not an easy task. In particular, classical upwind schemes
are not efficient. We are currently working on a characteristic scheme
which follows the one used in \cite{BBHV}. We will then include the
anti-angiogenic treatment in the equation, which mathematically means
transforming $G$ in a non-autonomous vector field. We will also
address the inverse problem and the parameter identification.
Our model has the good property that it has a small number of
parameters. So we hope that it can be used efficiently to make
predictions. We want to study mathematically the parameter
identification.

\appendix

\section{A short study of $\Wd(\Om)$}

Let
$$\Om=(1,b)\times(1,b),\; \Ga=\d\Om$$
and $G$ the vector field on $\Om$ with components given by \eqref{g1}
and \eqref{g2}.
Consider the space
$$\Wd(\Om):=\{V \in L^1(\Om)\,|\; \exists g \in L^1(\Om) \text{ s. t.
} \int_\Om V G\cdot\nabla
\phi = - \int_\Om g \phi, \; \forall \phi \in \C^1_c(\Om) \}$$
The function $g$ of this definition is denoted $\div(GV)$. We endow
this space with the norm
$$||V||_{\Wd}=||V||_{L^1} + ||\div(GV)||_{L^1}$$
With this norm, $\Wd(\Om)$ is a Banach space. In the following we also
denote this space by $\Wd^1(\Om)$.
\begin{rem}
If $\div(G)\in L^\infty$ and $V \in \Wd(\Om)$, we can define
$$G\cdot\nabla V := \div(GV) - V\div(G) \in L^1(\Om)$$
and the space $\Wd(\Om)$ is also the space of $L^1$ functions such
that there exists a function $g\in L^1(\Om)$ verifying
$$\int_\Om V \div(G \phi) = -\int_\Om g \phi, \quad \forall
\phi
\in \C^1_c(\Om)$$
\end{rem}
This space already appeared for the study of the boundary
problem for the transport equation (see
\cite{bardos,cessenat1,cessenat2}). \\
\espace In the same way, we define the space
$$\Wd^\infty(\Om)=\{U \in L^\infty(\Om)|\; G\cdot\nabla U \in
L^\infty(\Om)\}$$
\subsection{Conjugation of $\Wd(\Om)$ and $\W$}
For a function $V \in L^1$, the fact of belonging to $\Wd(\Om)$ means
that it is weakly
derivable along the characteristics. The next theorem makes this more
precise.

\begin{thm}[Conjugation of $\Wd(\Om)$ and $\W$]\label{conjugaisonWdiv}
Let $p=1$ or $\infty$. The spaces $\Wd^p(\Om)$ and
$W^{1,p}((0,+\infty);\,L^p(\Ga))$ are conjugated via $\Phi$ in the
following sense :
$$V\in \Wd^p(\Om) \Leftrightarrow (V \circ \Phi) |J_\Phi|^{1/p} \in
W^{1,p}((0,+\infty);\,L^p(\Ga))$$
Moreover, for $V \in \Wd^p(\Om)$ we have almost everywhere
$$\d_\tau (V\circ\Phi |J_\Phi|^{1/p})=\left\lbrace\begin{array}{cc}
(\div(GV)\circ \Phi) |J_\Phi|& if\; p=1 \\
(G\cdot \nabla V)\circ \Phi  & if\; p=\infty \end{array}\right.$$
and the application
$$\begin{array}{ccc} W^p_\div(\Om) & \rightarrow &
W^{1,p}((0,+\infty);\,L^p(\Ga))\\
                            V     & \mapsto     & V \circ \Phi
|J_\Phi|^{1/p} \end{array}$$
is an isometry.
\end{thm}

\begin{rem}\label{rem_loc}~\\
\espace $\bullet$ In particular, we deduce from the theorem applied to
the function $V=1$ that
$|J_\Phi| \in \W$ and we recognize the well known formula
$$\d\tau |J_\Phi|=\div(G)|J_\Phi|, \quad a.e$$
\espace $\bullet$ Since by proposition \ref{proprietes_flot}, we have
$|J_\Phi|^{-1} \in
W^{1,\infty}_{loc}(]0,\infty[\times \Ga^*)$ we deduce that
$$V(\Phi_\tau(\sig))=V(\Phi_\tau(\sig))|J_\Phi| \times |J_\Phi|^{-1}
\in
W_{loc}^{1,1}(]0,\infty[\,;\,L^1_{loc}(\Ga^*))$$
with
$$\d_\tau V(\Phi_\tau(\sig)) = G\cdot \nabla U (\Phi_\tau(\sig))$$
\end{rem}

\begin{proof}We first show the theorem on $\Wd^1(\Om)$ and then for
$\Wd^\infty(\Om)$\\
\espace $\bullet$ We prove now $(V \in \Wd(\Om)) \Rightarrow
(\V:=(V\circ \Phi) |J_\Phi| \in \W)$.
Let $V \in \Wd(\Om)$ and remark that $\V \in L^1(]0,\infty[\times
\Ga)$ since $|J_\Phi|$ is the
Jacobian of the change of variable between $\Om$ and
$]0,\infty[\times \Ga^*$. Then, using the definition of $\W$ we have
to prove that there exists a function $g \in L^1(]0,\infty[ \times
\Ga)$ such that for every
function $\tilda{\psi} \in \C^\infty_c(]0,\infty[)$
$$\int_0^\infty \V(\tau,\sig) \tilda{\psi}'(\tau)d\tau = -
\int_0^\infty g(\tau,\sig)
\tilda{\psi}(\tau)d\tau,\quad a.e.\; \sig \in \Ga.$$
As we aim to use the change of variable $\Phi$ which lives in
$]0,\infty[\times \Ga$, we will rather
prove that for every function $\zeta \in Lip(\Ga)$ (the Lipschitz
functions on $\Ga$)
\begin{align}\label{eqn_base}
\int_\Ga \left\lbrace\int_0^\infty \V(\tau,\sig) \tilda{\psi}'(\tau)
d\tau \right\rbrace \zeta(\sig)
d\sig  & =  \\
&\int_\Ga \left\lbrace- \int_0^\infty\div(GV)(\Phi_\tau(\sig))|J_\Phi|
\tilda{\psi}(\tau)
d\tau\right\rbrace
\zeta(\sig) d\sig \nonumber
\end{align}
which is sufficient to prove the result. Let now define the function
$$\psi(x,\tt):=\tilda{\psi}(\tau(x,\tt))$$
with $\tau(x,\tt)$ the time spent in $\Om$ defined in the section
\ref{model_ODE}. Then $\psi$ has compact support in $\Om$ and is
Lipschitz as the composition of a regular function
and a Lipschitz function (see prop. \ref{proprietes_flot} for the
locally Lipschitz regularity of
the function $(x,\tt) \mapsto \tau(x,\tt)$), thus differentiable
almost everywhere and the reverse formula
$\tilda{\psi}(\tau)=\psi(\Phi_\tau(1,1))$ (or $\psi(\Phi_\tau(\sig))$
for any $\sigma \in \Ga$ since the function $\psi$ depends only
on the time spent in $\Om$) yields
$$\tilda{\psi}'(\tau)=G\cdot\nabla\psi(\Phi_\tau(1,1)), \quad
a.e. \;\tau \in ]0,\infty[$$
since $\tau \mapsto \Phi_\tau(1,1)$ is $\C^1$. Doing now the change of
variables in the left hand
side of \eqref{eqn_base} yields
\begin{equation}\label{eqn_46}
\int_\Ga \left\lbrace\int_0^\infty \V(\tau,\sig)
\tilda{\psi}'(\tau) d\tau \right\rbrace \zeta(\sig)
d\sig=\int_\Om V(x,\tt) \zeta(\sig(x,\tt)) G\cdot \nabla \psi(x,\tt)
dxd\tt
\end{equation}
Still denoting $\zeta(x,\tt)$ the function $\zeta (\sig (x,\tt))$, we
remark that this function only
depends on the entrance point $\sig(x,\tt)$ and thus we have
$$(G\cdot \nabla \zeta) (\Phi_\tau(\sig))=\d_\tau
(\zeta(\Phi_\tau(\sig)))=\d_\tau(\zeta(\sig))=0,
\quad \forall \tau \geq 0,\; a.e \;\sig$$
To pursue the calculation, we need to regularize the Lipschitz
functions $\zeta$ and $\psi$ in order to use them in the
distributional definition of $\div(GV)$. We
use the following lemma, whose proof can be found in \cite{tartar},
p.60.
\begin{lem}\label{regularisation}
Let $f \in W^{1,\infty}(\Om)$ with $\Om$ a Lipschitz domain. Then
there exists a sequence $f_n \in \C^\infty(\barre{\Om})$ such that
$$f_n \xrightarrow{W^{1,p}} f \; \forall \, 1\leq p < \infty,\; f_n
\rightarrow f \; L^\infty \; weak-*, \;
\nabla f_n \rightarrow \nabla f \; L^\infty \; weak-* \;$$
\end{lem}
Now let $\psi_n \rightarrow \psi$ and $\zeta_m\rightarrow \zeta$ as in
the lemma. From the
demonstration of the lemma which is done by convolution with a
mollifier, since $\psi$ has compact support, so does
$\psi_n$ for n large enough. Now
remark that for each $n$ and $m$
$$G\cdot\nabla(\psi_n\zeta_m)=\zeta_m G\cdot\nabla \psi_n + \psi_n
G\cdot \nabla \zeta_m$$
The function $\psi_n\zeta_m$ is now valid in the distributional
definition of $\div(GV)$ and we have
\begin{align*}
\int_\Om V \zeta_m G\cdot\nabla\psi_n dxd\tt &=  \int_\Om V G\cdot
\nabla(\psi_n \zeta_m)dxd\tt -
\int_\Om V \psi_n G\cdot\nabla\zeta_m \\
							   & =
-\int_\Om \div(GV) \psi_n \zeta_m
dxd\tt - \int_\Om V \psi_n G\cdot\nabla\zeta_m dxd\tt
\end{align*}
Letting first $n$ going to infinity, then $m$ and
remembering that $G\cdot \nabla \zeta=0$ yields
$$\int_\Om V \zeta G\cdot\nabla\psi dxd\tt = -\int_\Om \div(GV) \psi
\zeta dxd\tt$$
Now doing back the change of variables $\Phi^{-1}$ gives
the identity
\eqref{eqn_base}. \\
\espace$\bullet$ We show now the reverse implication. Let $\V \in \W$
and $\psi \in
\C^\infty_c(\Om)$. Define $V(x,\tt):=(\V \circ
\Phi^{-1})|J_{\Phi^{-1}}|$ and
$\tilda{\psi}(\tau,\sig):=\psi(\Phi_\tau(\sig))$. Hence $\tilda{\psi}$
is $\C^1_c$ in the variable
$\tau$ and we have $\d_\tau \tilda{\psi} = (G\cdot \nabla \psi)\circ
\Phi$. Now
\begin{align*}
\int_\Om V G \cdot \nabla \psi dxd\tt & = \int_\Ga \int_0^\infty
\V(\tau,\sig) \d_\tau
\tilda{\psi}(\tau,\sig)d\tau d\sig = -\int_\Ga \int_0^\infty \d_\tau
\V
\tilda{\psi}d\tau d\sig \\
& = -\int_\Om \d_\tau
\V \circ \Phi^{-1} \psi dxd\tt
\end{align*}
Hence we have proved that $V\in \Wd(\Om)$ and that $\div(GV)=\d_\tau\V
\circ \Phi^{-1}$.\\
\espace $\bullet$ We prove now the part of the theorem on
$\Wd^\infty(\Om)$. Let $U\in \Wd^\infty(\Om)\subset \Wd(\Om)$. Then
$U\circ \Phi \in L^\infty(]0,\infty[\times \Ga)$. Moreover, following
the second point of the remark following the theorem, we have
$$\d_\tau U(\Phi_\tau(\sig))=G\cdot\nabla U(\Phi_\tau(\sig)) \in
L^\infty(]0,\infty[\times \Ga).$$
Using that for $\tilda{U}\in W^{1,\infty}((0,+\infty);L^\infty(\Ga))$
we have locally $G\cdot\nabla U:=\d_\tau \tilda{U} \circ \Phi^{-1}\in
L^\infty(\Om) $ with $U=\tilda{U}\circ \Phi^{-1}$ gives the reverse
implication.
\end{proof}
\subsection{Trace theorem, integration by part and calculus of
functions in $\Wd(\Om)$}
Thanks to the theorem \ref{conjugaisonWdiv}, we can now
transport the
theory of
vector-valued Sobolev spaces to $\Wd(\Om)$, for which we refer to
\cite{droniou1}.
\begin{prop}[Trace in $\Wd(\Om)$ and integration by
part]\label{trace_IPP}
Let $p=1$ or $\infty$ and $V\in \Wd^p(\Om)$. We
call \textbf{trace} of $V$
the following function
$$\gamma(V)(\sig)= (V\circ \Phi)(0,\sig), \quad \forall \, \sig \in
\Ga$$
We have $\gamma(V) G\cdot \nuu \in L^p(\Ga)$ and there exists $C>0$
such that
$$||\gamma(V) G\cdot \nuu||_{L^p(\Ga)} \leq C ||V||_{\Wd^p(\Om)},\quad
\forall V \in \Wd^p(\Om)$$
Moreover, if $V \in \Wd(\Om)$ and $U\in \Wd^\infty(\Om)$. Then
$$\int\int_\Om U \div(GV) + \int\int_\Om V G\cdot \nabla U = -\int_\Ga
\gamma(V)\gamma(U)
G\cdot\nuu$$
\end{prop}
\begin{proof}It is a direct consequence of the
properties of functions in $\W$ and in
$\Winf$ and the conjugation theorem \ref{conjugaisonWdiv}.
\end{proof}

 \begin{prop}\label{Wdiv_produit_Lipschitz} ~\\
\espace (i) Let $V \in \Wd(\Om)$ and $U \in \Wd^\infty(\Om)$. Then $UV
\in \Wd(\Om)$ and
$$\div(GVU)=V\left(G\cdot \nabla U\right) + U \div(GV)$$
\espace (ii) Let $H: \R \rightarrow \R$ a Lipschitz function and $V
\in \Wd(\Om)$. Then
$$H(V) \in \Wd(\Om)$$
and, almost everywhere
$$\div(G H(V))=H'(V) G\cdot \nabla V + H(V)\div(V)$$
\end{prop}

\begin{proof}(i) is a consequence of the product of a function in
$\W$ and a function in $\Winf$.\\
(ii) Let $H$ and $V$ satisfying the hypothesis. First remark
that $H$ being Lipschitz and $\Om$ bounded, the
function $H(V)$ is in $L^1(\Om)$. Now define
$\V(\tau,\sig)=V(\Phi_\tau(\sig))$. We will show that
$H(\V)|J_\Phi| \in \W$, in order to apply theorem
\ref{conjugaisonWdiv}. The function $\V$ is in
$W_{loc}^{1,1}(]0,\infty[\,;\,L^1_{loc}(\Ga^*))$ (see remark
\ref{rem_loc}). Thus it is absolutely
continuous in $\tau$ and $H$ being Lipschitz yields $H(\V)$ absolutely
continuous. Hence $H(\V)\in
W_{loc}^{1,1}(]0,\infty[\,;\,L^1_{loc}(\Ga^*))$. We conclude
the proof by using that
\begin{align*}
\d_\tau(H(\V)|J_\Phi|) & = \d_\tau(H(\V)) |J_\Phi| + \div (G) H(\V)
|J_\Phi| \\
                     & = H'(\V)\d_\tau\V |J_\Phi|
+\div (G)H(\V) |J_\Phi| \in L^1(]0,\infty[\times \Ga)
\end{align*}
 which requires the following lemma (see \cite{serrin}) to have
$\d_\tau(H(\V))= H'(\V)\d_\tau\V$.
\begin{lem}
Let H be a Lipschitz function, I a real interval, X a Banach space and
$u\in W^{1,1}(I\,;\,X)$. Then $H\circ u \in W^{1,1}(I\,;\,X)$, and
almost everywhere
$$(H\circ u)'=H'(u) u'$$
\end{lem}
\end{proof}

\bibliographystyle{plain}
\bibliography{biblio}

\subsection*{Acknowledgment}
The author would like to thank deeply the following people for
helpful discussions : Dominique
Barbolosi, Florence Hubert, Franck Boyer, Pierre Bousquet,
Thierry Gallouet and Vincent Calvez. He would like to address a
special thank to Assia Benabdallah for infallible support and great
attention.
\end{document}